\RequirePackage[l2tabu, orthodox]{nag}

\documentclass[12pt]{amsart}
\usepackage{fullpage,url,amssymb,enumerate}
\usepackage[all]{xy} 
\usepackage{comment}
\usepackage{graphicx}
\usepackage{breqn}

\usepackage[OT2,T1]{fontenc}

\usepackage{amsthm}

\usepackage{color}

\def\bC{\mathbb{C}}
\def\cC{\mathcal{C}}

\def\cI{\mathcal{I}}
\def\cL{\mathcal{L}}
\def\cM{\mathcal{M}}

\def\cO{\mathcal{O}}
\def\bP{\mathbb{P}}

\DeclareMathOperator{\Pic}{Pic}

\DeclareMathOperator{\Sym}{Sym}

\newtheorem{thm}{Theorem}[section]
\newtheorem{lem}[thm]{Lemma}
\newtheorem{conjecture}[thm]{Conjecture}
\newtheorem{cor}[thm]{Corollary}
\newtheorem{prop}[thm]{Proposition}

\newtheorem{question}{Question}

\newtheorem{defn}[thm]{Definition}

\makeatletter
\g@addto@macro\bfseries{\boldmath} 
\makeatother

\usepackage{microtype}

\usepackage[
	backref,
	pdfauthor={Carl Lian}, 
	pdftitle={An asymptotic Alexander-Hirschowitz theorem for surfaces},
]{hyperref}

\begin{document}

\title{An asymptotic Alexander-Hirschowitz theorem for surfaces}

\author{Carl Lian}
\address{Institut f\"{u}r Mathematik, Humboldt-Universit\"{a}t zu Berlin, 12489 Berlin, Germany}
\email{liancarl@hu-berlin.de}
\urladdr{\url{https://sites.google.com/view/carllian}}

\date{\today}

\begin{abstract}
Let $X$ be a smooth projective surface over $\mathbb{C}$ and let $\mathcal{L}$ be an ample line bundle on $X$. In this note, we show that, for all sufficiently large $d$, any number of general double points on $X$ imposes the expected number of conditions on the linear system $|\mathcal{L}^{\otimes d}|$. Equivalently, the space of $d$-plane sections of $X$ singular at any number of general points has the expected dimension. We conjecture that the same holds for $X$ of arbitrary dimension.
\end{abstract}

\maketitle


\section{Introduction}\label{intro}

\textcolor{red}{It was pointed out to the author soon after posting that a result subsuming both the main theorem and conjecture of this paper were proven in \cite{ah00}. This article is no longer intended for publication.}

We work throughout over $\bC$.

The Alexander-Hirschowitz theorem \cite{ah,bo} considers linear systems of degree $d$ hypersurfaces on $\bP^n$ constrained to be singular at a general collection of $k$ points. A na\"{i}ve parameter count predicts the dimension of this linear system to be
\begin{equation*}
\max\left(0,\binom{n+d}{d}-k(n+1)\right)-1.
\end{equation*}
Indeed, this holds except in a short list of exceptional cases.

\begin{thm}[Alexander-Hirschowitz]\label{AH_thm}
The linear system of degree $d$ hypersurfaces of $\bP^n$ singular at a general collection of $k$ points has the expected dimension, except when:
\begin{itemize}
\item $d=2$ and $2\le k\le n$,
\item $(n,d,k)=(2,4,5)$,
\item $(n,d,k)=(3,4,9)$,
\item $(n,d,k)=(4,3,7)$,
\item $(n,d,k)=(4,4,14)$.
\end{itemize}
\end{thm}

In particular, when $n$ is fixed and $d$ is sufficiently large, any number of general double points of $\bP^n$ impose the expected number of conditions on degree $d$ hypersurfaces. In this note, we consider this phenomenon on arbitrary smooth projective varieties.

\begin{defn}
Let $X$ be a smooth projective variety and let $\cL$ be a line bundle on $X$. We say that $(X,\cL)$ satisfies the \textbf{Alexander-Hirschowitz property (AH)} if, for all positive integers $k$, a general collection of double points $p_1,\ldots,p_k\in X$ imposes independent conditions on the linear system $|\cL|$. Equivalently, the linear system of divisors on $X$ singular at all of the $p_i$ has the expected dimension
\begin{equation*}
\max(0,\chi(X,\cL^{\otimes d})-k(\dim(X)+1))-1.
\end{equation*}
\end{defn}

When $\cL$ is ample, $\chi(X,\cL^{\otimes d})=h^0(X,\cL^{\otimes d})$ for all sufficiently large $d$. Moreover, for a \textit{fixed} integer $k$, Serre vanishing guarantees that the linear system of $d$-plane sections singular at $k$ points has expected dimension for all $d>d_0$ sufficiently large. However, for AH to hold, we require such a $d_0$ to exist \textit{independently} of $k$.

Our main result is:

\begin{thm}\label{thm_surfaces}
Suppose that $X$ is a surface and $\cL$ is ample. Then, $(X,\cL)$ satisfies AH.
\end{thm}

The natural conjecture is then:

\begin{conjecture}\label{main_conj}
Suppose that $X$ is a smooth projective variety of any dimension and $\cL$ is ample. Then, $(X,\cL)$ satisfies AH.
\end{conjecture}

To prove Theorem \ref{thm_surfaces}, we follow the same strategy for $\bP^2$ as explained in \cite{bo}. Namely, the Terracini Lemmas guarantee the existence of a multiple curve passing through the $p_k$ which must appear in the base locus of a linear system failing AH. One can immediately conclude when $X=\bP^2$, but in the case of arbitrary surfaces, we appeal to the Kawamata-Viehweg and Hodge index theorems to bound the dimension of the residual linear system in order to derive a contradiction.

It seems that the usual degeneration strategy for establishing the classical Alexander-Hirschowitz theorem will fail to prove Conjecture \ref{main_conj}, as one would need a way to establish base cases in $d$ for arbitrary $X$, which is achieved by explicit computation in the case of projective spaces. One could instead attempt to adapt the method in the case $\dim(X)=2$ for arbitrary $X$, leading to subtle geometric questions that may be of independent interest. We give a brief outlook on the higher-dimensional case in \S\ref{higher_dim_section}.

The notorious SHGH conjecture concerns linear systems of curves in $\bP^2$ constrained to pass through general points with higher multiplicities, see \cite{shgh} for a survey. One could similarly consider asymptotic versions of SHGH on arbitrary surfaces, but we do not pursue this direction here. 

Finally, we remark here that we also have the easier result that AH holds for curves.

\begin{prop}
Suppose that $X$ is a curve. Then, Conjecture \ref{main_conj} holds.
\end{prop}

\begin{proof}
Let $g$ be the genus of $X$. We may assume that $d>2g-2$, so that $h^0(X,\cL)=d-g+1$. The linear system $|\cL(-2(p_1+\cdots+p_k))|$ will clearly have expected dimension if $d<2k$ or $2k<d-(2g-2)$. Moreover, when $k\ge g$, any line bundle of degree $d-2k$ on $X$ is of the form $\cL(-2(p_1+\cdots+p_k))$, and the generic such line bundle has vanishing $H^1$ if $d-2k\ge0$, so in this case $|\cL(-2(p_1+\cdots+p_k))|$ also has expected dimension. Thus, a collection of $k$ general double points can only fail to impose independent conditions on $|\cL^{\otimes d}|$ for $k<g$, and for such $k$, AH can only fail for finitely many $d$.
\end{proof}

\subsection{Acknowledgments}

I thank Amol Aggarwal, Joe Harris, Dennis Tseng, and especially Johan de Jong for discussions related to this topic. This work was completed with the support of NSF Graduate Research and Postdoctoral Fellowships.

\section{Proof of Theorem \ref{thm_surfaces}}

We first recall the Terracini Lemmas, which are also used to prove the Alexander-Hirschowitz theorem for $\bP^2$ and date to \cite{terracini}. The proofs for general arbitrary $X$ are essentially identical to those given in \cite{bo}; we include them in full for the reader's convenience.

\begin{lem}[First Terracini Lemma]
Let $X\subset\bP^N$ be a non-degenerate closed subvariety, and let $x_1,\ldots,x_k$ be general points such that the linear span $\langle x_1,\ldots,x_k\rangle$ has dimension $k-1$ (in particular, $k\le N+1$). Let $z\in \langle x_1,\ldots,x_k\rangle$ be a general point. Then,
\begin{equation*}
T_{z}\sigma_k(X)=\langle T_{x_1}X,\ldots,T_{x_k}X\rangle,
\end{equation*}
where $\sigma_k(X)$ denotes the $k$-secant variety of $X$.
\end{lem}

\begin{proof}
We recall the construction of $\sigma_k(X)$: let $Y\subset X^k\times\bP^N$ be the locus of $(x_1,\ldots,x_k,z)$ such that $\langle x_1,\ldots,x_k\rangle$ has dimension $k-1$ and $z\in\langle x_1,\ldots,x_k\rangle$. Then, $\sigma_k(X)$ is the closure of the image of $Y$ upon projection to $\bP^N$. The composite $Y\subset X^k\times\bP^N\to X^k$ is a $(k-1)$-plane bundle over the open locus in $X^k$ of points in linearly general position, so $Y$ is integral, and the map $\varphi:Y\to\sigma_k(X)$ is generically smooth. In particular, at a general point $(x_1,\ldots,x_k,z)\in Y$, $\varphi$ is surjective on tangent spaces.

Using the description of $Y$ as a $(k-1)$-plane bundle over $X^k$ in a neighborhood of $(x_1,\ldots,x_k)$, we have
\begin{equation*}
T_{(x_1,\ldots,x_k,z)}Y=T_{x_1}X\oplus\cdots T_{x_k}X\oplus T_{z}\langle x_1,\ldots,x_k\rangle,
\end{equation*}
the image of which in $T_{z}\sigma_k(X)$ is exactly $\langle T_{x_1}X,\ldots,T_{x_k}X\rangle$. By assumption, $(T\varphi)_{(x_1,\ldots,x_k,z)}$ is surjective, so the proof is complete.
\end{proof}

Let $X$ be a smooth projective variety and $\cL$ be a very ample line bundle inducing the closed embedding $i:X\to\bP^n$, where $n=h^0(X,\cL)-1$. Let $\cI_X$ be the corresponding ideal sheaf on $\bP^n$; for all sufficiently large $d$, we have that $H^1(\bP^n,\cI_X(d))=0$, and that the restriction map
\begin{equation*}
H^0(\bP^n,\cO(d))\to H^0(X,\cL^{\otimes d})
\end{equation*}
is surjective. Let $j:\bP^n\to\bP^N$ be the $d$-th Veronese embedding. In particular, we have
\begin{equation*}
H^0(\bP^N,\cO(1))=H^0(\bP^n,\cO(d)).
\end{equation*}

\begin{lem}[Second Terracini Lemma]
Let $X,\cL,d$ be as above, defining closed embeddings $X\subset\bP^n\subset\bP^N$. Let $p_1,\ldots,p_k\in X$ be general points spanning a $(k-1)$-plane, and suppose that double points at the $p_i$ fail to impose independent conditions on the linear system $|\cL^{\otimes d}|$. Then, there exists a subscheme $C\subset X$, all of whose components are positive-dimensional, passing through the $p_i$ such that for all $p\in C$, we have $T_p X\subset\langle T_{p_1}X,\ldots, T_{p_k}X\rangle$, where the tangent spaces are regarded as linear subspaces of $\bP^N$.
\end{lem}

\begin{proof}
A divisor $D\in|\cL^{\otimes d}|$ is the restriction of a hyperplane $H$ in $\bP^N$, and $D$ is singular at $p\in X$ if and only if $T_{p}X\subset H$. Thus, double points at $p_i$ impose independent conditions for $|\cL^{\otimes d}|$ if and only if the linear span $\langle T_{p_1}X,\ldots, T_{p_k}X\rangle$ has the expected dimension in $\bP^N$. Let $z\in\langle p_1,\ldots,p_k\rangle$ be a general point. Then, by the first Terracini Lemma, we have $T_z\sigma_k(X)=\langle T_{p_1}X,\ldots, T_{p_k}X\rangle$, where the secant variety is taken in $\bP^N$. Thus, if double points at general $p_i$ fail to impose independent conditions on $|\cL^{\otimes d}|$, then $T_z\sigma_k(X)$ fails to have expected dimension for a general point $z\in\sigma_k(X)$. In particular, the secant variety $\sigma_k(X)$ fails to have expected dimension.

Let $Y\subset X^k\times\bP^N$ be the locus of $(x_1,\ldots,x_k,z)$ such that $\langle x_1,\ldots,x_k\rangle$ has dimension $k-1$ and $z\in\langle x_1,\ldots,x_k\rangle$. Then, $\sigma_k(X)$ is the closure of the image of $Y$ upon projection to $\bP^N$, and if $\sigma_k(X)$ fails to have expected dimension, then the general fiber of the projection map $\varphi:Y\to\sigma_k(X)$ has positive dimension. Note that $\varphi$ is invariant under the action of $S_k$ on $X^k$ permuting the factors.

Now, let $(x_1,\ldots,x_k,z)\in Y$ be a general point such that the fiber of $C_z=\varphi^{-1}(z)$ is positive-dimensional. Let $C$ be the closure of the image of $C_z$ upon projection to any of the $k$ factors of $X$. We have $(p_1,\ldots,p_k,z)\in C_z$, so $p_i\in C$ for each $i$. Moreover, if $(q_1,\ldots,q_k,z)\in C_z$, then $\dim\langle q_1,\ldots,q_k\rangle=k-1$, and $T_{q_1}X\subset T_{z}\sigma_k(X)=\langle T_{p_1}X,\ldots, T_{p_k}X\rangle$. The desired property of $C$ follows for $p\in C$ general, and thus for all $p\in C$.
\end{proof}

\begin{cor}
Let $C$ be as above. Then, any divisor in $|\cL^{\otimes d}|$ singular at the $p_i$ is singular along $C$.
\end{cor}

\begin{proof}
Let $D=H\cap X\in|\cL^{\otimes d}|$, where $H\subset \bP^N$ is a hyperplane. For any $p\in C$, we have $T_p X\subset\langle T_{p_1}X,\ldots, T_{p_k}X\rangle\subset H$, so $D=H\cap X$ is also singular at $p$. 
\end{proof}

\begin{proof}[Proof of Theorem \ref{thm_surfaces}]
Fix a smooth projective surface $X$ and an ample line bundle $\cL$ on $X$. In fact, we may assume $\cL$ is very ample, defining a closed embedding $X\to\bP^N$, and that $d$ is large enough so that $H^1(\bP^N,\mathcal{I}_X(d))=0$.

Suppose that $k$ general double points fail to impose independent conditions on $|\cL^{\otimes d}|$. We assume $d$ is large enough that $h^0(X,\cL^{\otimes d})\ge 5$. Note that we may assume that $k\lesssim h^0(\cL^{\otimes d})/3$, so that a general collection of $k$ points on $X$ spans a $(k-1)$-plane under the embedding defined by $|\cL^{\otimes d}|$. By the Second Terracini Lemma, $k$ general points contain a curve in some linear system $|\cM|$ satisfying the property that $\cL^{\otimes d}\otimes\cM^{\otimes-2}$ is effective; let $S(\cL,d)$ be the set of effective line bundles $\cM$ with this property. 

There are finitely many components of $\Pic(X)$ containing some $\cM\in S(\cL,d)$. Indeed, curves appearing in the linear systems $|\cL|$ have bounded Hilbert polynomial, and thus appear in a bounded family; the same is therefore true for their underlying line bundles. Let $P$ be the union of these components. Then, $P$ is a smooth $\bC$-scheme of dimension $q=H^1(X,\cO_X)$.

Let $H$ be the Chow variety of curves in the linear systems parametrized by $P$, so that there is a forgetful morphism $H\to P$ whose fiber over $[\cM]$ is $|\cM|=\bP H^0(X,\cM)$. Let $\cC\to H$ be the universal curve. We then have a canonical morphism
\begin{equation*}
\cC\times_{H}\cdots\times_H\cC\to\Sym^k(S),
\end{equation*}
which, by assumption, is dominant. Comparing dimensions, we must have
\begin{equation*}
2k\le\dim H+k\le q+\max_{\cM\in S(\cL,d)}(h^0(X,\cM)-1)+k
\end{equation*}
so
\begin{equation*}
k\le q-1+\max_{\cM\in S(\cL,d)}h^0(X,\cM).
\end{equation*}

On the other hand, we may assume that $k$ is as large as possible so that the expected dimension of $|\cL^{\otimes d}-2(p_1+\cdots+p_k)|$ is at least $-1$. Thus, we also have the inequality
\begin{equation*}
k\ge\frac{h^0(X,\cL^{\otimes d})-2}{3}\ge1.
\end{equation*}

Thus, if AH fails for the linear system $(X,\cL^{\otimes d})$, then there exists a line bundle $\cM\in S(\cL,d)$ such that
\begin{equation}\label{main_ineq}
q+h^0(X,\cM)-1\ge\frac{h^0(X,\cL^{\otimes d})-2}{3}.
\end{equation}

Consider the set of curves in the algebraic equivalence class of $\cM$. We may assume these do not have one-dimensional common base locus: indeed, suppose that every member of a linear system in the algebraic equivalence class of $\cM$ contains some fixed integral curve $E$ as a component. Then, by simply removing $E$, a general collection of $k$ points disjoint from $E$ contains a curve in $\cM(-E)$, and thus, we may replace $\cM$ with $\cM(-E)$, preserving all of the needed properties.

\begin{lem}\label{big_and_nef}
The line bundle $\cM\otimes\cL$ is big and nef.
\end{lem}

\begin{proof}
The bigness is immediate, because $\cM$ is effective and $\cL$ is ample. Now, suppose that $\cM\cdot E<0$ for some integral effective curve $E$. Then, $E$ must be contained in the common base locus of the algebraic equivalence class of $\cM$, which we have assumed to be zero-dimensional. Thus, $\cM$ is nef, and $\cM\otimes\cL$ is nef as well.
\end{proof}

By Kawamata-Viehweg vanishing \cite{kawamata,viehweg}, we have $h^1(X,(\cM\otimes \cL)\otimes K_X)=0$. We may write $\cL\otimes K_X\cong\cO_X(A_2-A_1)$, where $A_1,A_2$ are fixed very ample divisors; we may additionally assume $A_1,A_2$ are smooth and connected. Then, from the exact sequence
\begin{equation*}
0= H^1(X,\cM(A_2-A_1))\to H^1(X,\cM(A_2))\to H^1(A_1,\cM(A_2)|_{A_1}),
\end{equation*}
we find that $h^1(\cM(A_2))$ is bounded above by a constant. Therefore,
\begin{equation*}
h^0(\cM(A_2))\le \chi(\cM(A_2))+O(1).
\end{equation*}
Let $M$ be a divisor with underlying line bundle $\cM$. We then have, by Riemann-Roch,
\begin{align*}
\chi(\cM(A_2))-\chi(\cM)&=\frac{1}{2}[(M+A_2)(M+A_2-K_X)-M(M-K_X)]\\
&=\frac{1}{2}[A_2^2+A_2(2M-K_X)]\\
&=\gamma'+A_2\cdot M,
\end{align*}
where $\gamma'$ is a constant. On the other hand, because $A_2$ is ample and $dL-2M$ is effective, where $L$ is the divisor class of $\cL$, we have
\begin{equation*}
M\cdot A_2<\frac{d}{2}(L\cdot A_2),
\end{equation*}
so combining with the above, we get
\begin{equation*}
\chi(\cM(A_2))-\chi(\cM)\le O(d).
\end{equation*}

We now estimate $\chi(\cM)$. By Riemann-Roch,
\begin{equation*}
\chi(\cM)=\chi(\cO_X)+\frac{1}{2}M(M-K_X).
\end{equation*}
By the Hodge Index Theorem, we have
\begin{equation*}
M^2<\frac{(L\cdot M)^2}{L^2}<\frac{\left(\frac{d}{2}L^2\right)^2}{L^2}=\frac{L^2}{4}d^2
\end{equation*}

Furthermore,
\begin{align*}
M(-K_X)&<M\cdot A_1+M\cdot L\\
&\le O(d),
\end{align*}
so $\chi(\cM)\le\frac{d^2}{8}L^2+O(d)$. Combining, we have
\begin{align*}
h^0(\cM)&\le h^0(\cM(A_2))\\
&\le\chi(\cM(A_2))+O(1)\\
&\le \chi(\cM)+O(d)\\
&\le\frac{L^2}{8}d^2+O(d).
\end{align*}
Because $h^0(X,\cL^{\otimes d})=\frac{L^2}{2}d^2+O(d)$, (\ref{main_ineq}) must fail for all $d$ sufficiently large. Therefore, for such $d$, AH holds.

\end{proof}

\section{Higher dimensions}\label{higher_dim_section}

The usual strategy for proving the Alexander-Hirschowitz theorem for $\bP^n$, as originally implemented in full in \cite{ah} and simpified in \cite{bo}, involves degenerating the $k$ points so that some of them lie on a fixed hyperplane, and reducing to the cases of degree $d$ hypersurfaces in $\bP^n$ and degree $(d-1)$ hypersurfaces in $\bP^n$. While the degeneration argument seems likely to work in principle for an arbitrary ambient variety, one would need a way to establish a base case for one value of $d$ and a fixed $X$. This is done by explicit calculation for cubics on $\bP^n$, but such a direct approach is unavailable for general $X$.

Instead, one can attempt to extend the strategy for surfaces to the higher-dimensional case. By the Terracini Lemmas, a linear system for which AH fails will still have a multiple base curve $C$, and varying the $k$ points gives a family of such $C$. Therefore, there would exist a family of curves passing through $k$ general points on $X$, leading to the question:
\begin{question}\label{question_hilbert}
Suppose that $k$ general points on $X$ lie on a curve parametrized by some component of the Hilbert scheme of degree $e$ curves in $X$. What is the smallest possible value of $e$?
\end{question}

On the other hand, $C$ would have to appear in the singular locus of a hyperplane section of degree $d$ on $X$. so one also arrives at the second question:
\begin{question}\label{question_singular_locus}
Suppose that a $d$-plane section of $X$ is singular along a curve of degree $e$. What is the largest possible value of $e$?
\end{question}

Even in the case of $\bP^3$, however, the sharpest possible independent bounds for Questions \ref{question_hilbert} and \ref{question_singular_locus} do not seem yield a contradiction. Thus, more refined information about $C$ would presumably need to be taken into account.

%
%
%
%

\end{document}